\documentclass[10pt,a4paper]{amsart}
\usepackage{graphicx, amssymb, xcolor,pdfsync}
\usepackage[all,cmtip]{xy}
\numberwithin{equation}{section}
\usepackage{mathrsfs}
\usepackage{tikz}
\usepackage{tikz-cd}
\usepackage{faktor}
\usepackage{setspace}
\usepackage{hyperref}
\usepackage{graphicx}
\usepackage{yhmath}
\usepackage{mathdots}
\usepackage{MnSymbol}
\newcommand{\m}{\mbox{ mod }}
\hypersetup{
    colorlinks=true,       
    linkcolor=blue,          
    citecolor=blue,        
    filecolor=blue,      
    urlcolor=blue           
}
\usepackage{tikz}
\usetikzlibrary{matrix,arrows}
\DeclareMathOperator{\Aut}{Aut}
\usepackage[switch]{lineno}
\usepackage{yfonts}

\advance\hoffset by -0.9 truecm

\begin{document}

\newcommand{\s}{\vspace{0.2cm}}
\newcommand{\Sp}{\mbox{Sp}}
\newcommand{\sy}{\mbox{sym}}
\newcommand{\GP}{\mathcal{G}_p}
\newcommand{\FP}{\mathcal{F}_p}
\newcommand{\G}{\mathbb{Z}_p^2 \rtimes \mathbf{D}_3}
\newcommand{\Z}{\mathbb{Z}}
\newcommand{\D}{\mathbf{D}}
\newcommand{\Q}{\mathcal{Q}}
\newcommand{\HH}{\mathbb{H}}
\newcommand{\oph}{orientation-preserving homeomorphism}

\newtheorem{theo}{Theorem}
\newtheorem{prop}{Proposition}
\newtheorem{coro}{Corollary}
\newtheorem{lemm}{Lemma}
\newtheorem{claim}{Claim}
\newtheorem{conj}{Conjecture}
\newtheorem{example}{Example}
\theoremstyle{remark}
\newtheorem{rema}{\bf Remark}
\newtheorem*{rema1}{\bf Remark}
\newtheorem*{defi}{\bf Definition}
\newtheorem*{theo*}{\bf Theorem}
\newtheorem*{coro*}{Corollary}
\newtheorem*{conj*}{Conjecture}
\newtheorem*{theoA}{\bf Theorem A}
\newtheorem*{theoB}{\bf Theorem B}
\newtheorem*{theoC}{\bf Theorem C}
\newtheorem*{theoD}{\bf Theorem D}
\newtheorem*{coroA1}{\bf Corollary A1}
\newtheorem*{coroA2}{\bf Corollary A2}
\newtheorem*{coroB1}{\bf Corollary B1}
\newtheorem*{coroB2}{\bf Corollary B2}
\newtheorem*{coroC1}{\bf Corollary C1}
\newtheorem*{coroC2}{\bf Corollary C2}
\newtheorem*{coroD1}{\bf Corollary D1}
\newtheorem*{coroD2}{\bf Corollary D2}

\title{On orientably-regular maps of Euler characteristic $-2p^2$}
\date{}

\author{Tom\'as Foncea E. and Sebasti\'an Reyes-Carocca}
\address{Departamento de Matem\'aticas, Facultad de Ciencias, Universidad de Chile, Las Palmeras 3425, Santiago, Chile}
\email{tomas.foncea@ug.uchile.cl, sebastianreyes.c@uchile.cl}

\thanks{Partially supported by ANID Fondecyt Regular  Grants 1220099 and 1230708}

\begin{abstract} 
In this article, we study orientably-regular maps of Euler characteristic $-2p^2$ and classify those that admit a group of orientation-preserving automorphisms of order $10p^2$, where $p$ is a prime number. Along the way, we classify all compact Riemann surfaces (or complex algebraic curves) of genus $1+p^2$ endowed with a group of conformal automorphisms of order $5p^2$.
\end{abstract}
\maketitle
\thispagestyle{empty}

\section{Introduction and statement of the results} 

The study of maps on compact orientable real surfaces represents a classical and rich area of research, bridging algebra, geometry and topology. The origins of this area go back to the ancient Greeks, as maps generalise the classically known Platonic solids. We refer the reader to the survey article by Širáň \cite{Si13} for a comprehensive overview of this topic. The theory has subsequently been extended to include surfaces that may be non-orientable and also with nonempty boundary; see, for instance, \cite{BS85}, \cite{JS94} and \cite{IS94}. The broad problem of classifying such objects is at the core of important and recent developments, and has attracted considerable interest over the past decades. This classification can be achieved by considering different features. For instance, by considering maps satisfying some prescribed graph-theoretical property, or whose automorphism groups share a given group-theoretical property. Besides, from a topological point of view, a central problem is that of classifying all maps on real surfaces of a given Euler characteristic.

\s

In this article, we restrict attention to orientably-regular maps, meaning that the supporting compact surface is orientable, and the orientation-preserving automorphism group of the map acts transitively on the set of edges of the graph. The classification of orientably-regular maps of a given low genus was already known in the late sixties; see, for instance  \cite{G69}, \cite{Sh59} and \cite{T32} for genus $2,3,4,5$ and $6$. The advent of significant advancements in computer algebra systems and finite group theory brought about a renewed interest in the classification of these objects. Indeed, Conder and Dobcs\'anyi succeeded in obtaining in \cite{CP01} a computer-assisted classification of all orientably-regular maps for genus 2 to 15. These methods were taken further in the subsequent paper \cite{C09}, completing the classification up to genus 101. We refer to Conder's webpage \cite{C} for several --and updated-- databases.  

\s

Once the developments described above for lower genera were achieved, researchers began to address the problem of classifying orientably-regular maps embedded on surfaces of infinitely many genera. Although the general problem of determining all such maps and the algebraic structure of their automorphism groups is challenging, it tends to be more tractable when the Euler characteristic $2(1 - g)$ of the supporting surface has a simple numerical form. The simplest case to deal with is therefore that in which $g - 1$ is a prime number. By means of geometric arguments and in terms of the language of compact Riemann surfaces, Belolipetsky and Jones in \cite{BJ05} succeeded in classifying all orientably-regular maps of Euler characteristic $-2p$ with orientation-preserving automorphism groups of order $\lambda p$, where $\lambda >6$ is an integer and $p$ is a (sufficiently large) prime number.  In \cite{IJRC} these results were extended by filling in some gaps and in this way giving a complete classification in that case. On the other hand, after extending the methods employed in \cite{BNS05}, Conder, {\v S}ir\'a{\v n} and Tucker in \cite{CST10} went further by obtaining a classification of orientably-regular maps of Euler characteristic $-2p$ for any prime $p \geqslant 17$, mixing up purely group-theoretic arguments with computer-assisted programs (see also the recent article \cite{BCSS24}). 

\s

For recent articles dealing with problems of a similar nature, we refer the reader to \cite{C23}, \cite{CNS12}, \cite{CS20}, \cite{Du23}, \cite{Gill13}, \cite{nos17}, \cite{Hu19}, \cite{T24} and the references therein.

\s

A natural generalisation of the aforementioned results is that of classifying orientably-regular maps on surfaces of Euler characteristic $-2p^n$ or, equivalently, on surfaces of genus $1+p^n$, where $n \geqslant 2$ is an integer and $p$ is a prime number. Motivated by the study of regular maps with simple underlying graphs, the first step toward solving this broad problem was taken by Ma in  \cite{Ma21}. Concretely, he studied orientably-regular maps $\mathscr{M}$ on surfaces of Euler characteristic $-2p^2,$ where $p$ is an odd prime number, and provided a classification, in terms of the type, of all those having orientation-preserving automorphism group of order divisible by $p.$ After disregarding nine exceptional cases for small primes, he proved that necessarily $$|\mbox{Aut}^+(\mathscr{M})|=\lambda p^2 \, \mbox{ where } \,\, \lambda \in \{8,10,12,16,24,48\}.$$Conversely, Ma proved the existence of an orientably-regular map for each $\lambda$ as before, and gave conditions for $p$ under which such maps exist. For instance, in \cite[Lemma 2.15]{Ma21} it was proved that the case $\lambda=10$ is realised if and only if $p \equiv \pm 1 \mbox{ mod }5.$ Nevertheless, the problem of determining both the exact number and the algebraic structure of the orientation-preserving automorphism group of such maps has not been addressed previously. This article is aimed at answering these two problems for the case $\lambda=10$.

\s

We anticipate that this problem is intimately related to that of classifying compact Riemann surfaces (or, equivalently, complex algebraic curves) of genus $1+p^2$ admitting a group of conformal automorphisms of order $5p^2$. In this sense, our results can be viewed as a natural extension of those in \cite{CRC21}, which addresses the classification in the case of $3p^2$ conformal automorphisms. 

Our main result is the following.

\s

{\bf Theorem.} Let $p \geqslant 11$ be a prime number. 

\s

{\bf 1.} If $p \equiv -1 \mbox{ mod } 5$ then there are exactly two orientably-regular maps of Euler characteristic $-2p^2$ with orientation-preserving automorphism group of order $10p^2.$ These maps are reflexive. If $t_1$ and $t_2$ are the solutions of $y^2+y-1 \equiv 0 \mbox{ mod }p$ then the orientation-preserving automorphism group of them is isomorphic to $$
\langle a, b, c, d :  a^p = b^p = c^5 = d^2 = [a,b]= (da)^2=(db)^2=1, cac^{-1} = b^{t_2}, cbc^{-1} = (ab)^{t_1}, 
 dcd = a^{t_1^2} b^2c \rangle.$$

{\bf 2.} If $p \equiv 1 \mbox{ mod } 5$ then there are exactly ten orientably-regular maps of Euler characteristic $-2p^2$ with orientation-preserving automorphism group of order $10p^2.$ Two of them are reflexive, and the remaining ones form four chiral pairs. The orientation-preserving automorphism group of two pairs is isomorphic to 
$$\langle a,b,c : a^{p^2} =b^5 =c^2 =(ca)^2=[c,b]=1,bab^{-1} =a^k\rangle,$$ where $k$ is a fifth primitive root of unity in $\mathbb{Z}_{p^2};$ the orientation-preserving automorphism group of two pairs is   isomorphic to$$
\langle a,b,c, d: a^p = b^p = c^5 = d^2 = [a,b]=(da)^2=(db)^2=1,cac^{-1} = a^s,cbc^{-1} = b^{s^2}, dcd = a^{1-s}b^{1-s^2} c \rangle;$$and the orientation-preserving  automorphism group of the reflexive ones is isomorphic to $$
 \langle a,b,c, d: a^p = b^p = c^5 = d^2 = [a,b]=(da)^2=(db)^2=1,cac^{-1} = a^s,
cbc^{-1} = b^{s^4}, dcd = a^{1-s}b^{1-s^4} c \rangle,$$where $s$ is a fifth primitive root of unity in $\mathbb{Z}_{p}.$

\s

The strategy we employ to prove the main result relies on the close relationship between orientably-regular maps, triangle Fuchsian groups, surface-kernel epimorphisms and compact Riemann surfaces. Along the way, we shall obtain some interrelated results concerning orientably-regular hypermaps ({\it dessins d'enfants} in Grothendieck's terminology) and Riemann surfaces which are interesting in their own right; we mention them here. 

\s

{\bf (a) } We obtain a classification of all orientably-regular hypermaps of Euler characteristic $-2p^2$ with orientation-preserving automorphism group of order $5p^2$. We also determine their automorphism groups and the exact number of the corresponding supporting surfaces.

\s

{\bf (b) } We give a classification of all those compact Riemann surfaces of genus $1+p^2$ endowed with a group of conformal automorphisms of order $5p^2.$ We also determine the full automorphism group of such Riemann surfaces and show that they are non-hyperelliptic. 

\s

{\bf (c) } We then deduce  that there is no compact Riemann surface of genus $1+p^2$ with exactly $5p^2$ conformal  automorphisms (c.f. \cite[Theorem 1]{IRC}).

\s

This paper is organised as follows. In \S\ref{preli} we shall introduce some notations, and discuss briefly the relationship between orientably-regular maps,  orientably-regular hypermaps and Riemann surfaces. In \S\ref{ls} we shall establish three auxiliary lemmata concerning group actions on Riemann surfaces. The proof of the main result will be a direct consequence of the results proved in \S\ref{resultados}. 
 We end the article by pointing out some remaks in \S\ref{remass}.
\section{Riemann surfaces and orientably-regular (hyper)maps}\label{preli}

\subsection*{Belyi pairs} It is well-known that the category of compact Riemann surfaces is equivalent to that of complex algebraic curves. An interesting problem to consider in this framework consists of the description of all those compact Riemann surfaces that are defined over a given subfield $\mathbb{K}$ of the field of complex numbers, namely, those Riemann surfaces that are isomorphic to complex algebraic curves defined by polynomials with coefficients in $\mathbb{K}$. A consequence of the classical Weil's theorem \cite{W56} is the fact that a sufficient condition for a Riemann surface $S$ to be defined over the field of algebraic numbers $\bar{\mathbb{Q}}$ is the existence of a non-constant holomorphic function $$\beta : S \to \mathbb{P}^1 \mbox{ with three critical values.}$$Conversely, Belyi proved in \cite{B79} that this condition is also necessary, and thus opening several interesting directions for further investigation. The pair $(S, \beta)$ is called a {\it Belyi pair}. 

\subsection*{Fuchsian description of Belyi pairs} Let $\mathbb{H}$ denote the upper half-plane, and let $S$ be a compact Riemann surface of genus greater than one. The classical uniformisation theorem states that there is a co-compact Fuchsian group $\Gamma$ such that $S \cong \mathbb{H}/\Gamma.$ Moreover, if $(S, \beta)$ is a Belyi pair then $\beta$ induces a finite index inclusion $\Gamma \leqslant \Delta$ of Fuchsian groups, where \begin{equation*}\label{tri}\Delta=\Delta(n,m,k)=\langle x,y,z : x^n = y^m=z^k=xyz=1 \rangle \mbox{ for }n,m,k \in \mathbb{Z}_{\geqslant 2} \mbox{ such that }\tfrac{1}{n}+\tfrac{1}{m}+\tfrac{1}{k}< 1.\end{equation*}The group $\Delta$  is called a {\it triangle} Fuchsian group of {\it signature} $(n,m,k).$ A Belyi pair $(S, \beta)$ is called {\it regular} if the inclusion $\Gamma \leqslant \Delta$ is normal or, equivalently, if $S$ is endowed with a group of conformal automorphisms $G\cong \Delta/\Gamma$ such that $\beta$ is (equivalent to) the branched regular covering map $S \to S/G$ given by the action of $G$ on $S$. This is, in turn, equivalent to the existence of a group epimorphism $$\theta: \Delta \to G \mbox{ such that }\Gamma=\mbox{ker}(\theta).$$We say that $\theta$ is a {\it surface-kernel epimorphism} (hereafter, {\it ske} for short) and that it represents the action of $G$ on $S.$ Two ske are called {\it $G$-equivalent} if they differ by an automorphism of $G.$ In such a case, the Riemann surfaces uniformised by their kernels agree. Observe that $S/G \cong \mathbb{P}^1$ and that the action of $G$ on $S$ has precisely three short orbits corresponding to the $\beta$-fibers of the three critical values of $\beta$, that we may assume to be $\infty, 0$ and $1.$ Notice, moreover, that the $G$-stabilisers of the points lying in these short orbits have orders $n,m$ and $k$ respectively. We also say that the action of $G$ on $S$ has {\it signature} $(n,m,k)$ and that the regular Belyi pair is of {\it type} $(n,m,k).$ 

\subsection*{Orientably-regular maps} Let $X$ be a compact orientable surface. A {\it map} $\mathscr{M}$ on $X$ is an embedding of a finite connected graph $\mathscr{G}$ such that $S-\mathscr{G}$ consists of a disjoint union of topological discs; such discs are called the {\it faces} of $\mathscr{M}$. The genus and the Euler characteristic of the map is the genus and the Euler characteristic of the supporting surface. An automorphism of $\mathscr{M}$ is an incidence-preserving permutation of its edges, vertices and faces induced by homeomorphisms of the underlying surface. 
 A map is called {\it orientably-regular} if its group of orientation-preserving automorphisms acts regularly (and hence transitively) on the set of its edges. If, in addition, the map $\mathscr{M}$ admits an orientation-reserving automorphism then $\mathscr{M}$ is called {\it reflexive}; otherwise, it is called {\it chiral} and $(\mathscr{M}, \bar{\mathscr{M}})$ is a {\it chiral pair}. The {\it type} of an orientably-regular map is the pair $\{m,k\},$ where $k$ is the degree of each vertex and $m$ is the size of each face.

\subsection*{Regular Belyi pairs and orientably-regular hypermaps}

Let $(S, \beta)$ be a Belyi pair and assume the critical values of $\beta$ to be $\infty, 0$ and $1$. Then $\mathscr{D}:=\beta^{-1}([0,1])$ is map whose vertices are the preimages of $0$ and $1$. Furthermore, after colouring the vertices of $\mathscr{D}$ in two different ways, depending on whether they are preimages of $0$ or $1$, we obtain a bipartite map, also called a {\it dessin d'enfant}. Conversely, each dessin d'enfant $\mathscr{D}$ on a compact orientable surface $X$ yields a structure of compact Riemann surface $S$ on $X$ and a holomorphic function $\beta: S \to \mathbb{P}^1$ such that $(S, \beta)$ is a Belyi pair and $\mathscr{D}=\beta^{-1}([0,1])$. In this way, there is a remarkable correspondence between Belyi pairs, complex algebraic curves defined over $\bar{\mathbb{Q}}$ and dessins d'enfant. The type of the dessin d'enfant is  $\{n,m,k\}$ if the Belyi pair $(S, \beta)$ is of type $(n,m,k)$. Walsh proved in \cite{W75} that there is a bijective correspondence between bipartite maps and {\it hypermaps} of the same genus. We recall that hypermaps were formally introduced in the seventies by Cori in \cite{Cori75} as algebraic and combinatorial structures that model embeddings of hypergraphs on orientable surfaces, and that they generalise maps on surfaces. Concretely, maps of type $\{m,k\}$ are precisely hypermaps of type $\{2, m,k\},$ and therefore they are in correspondence with the finite index subgroups of triangle Fuchsian groups of signature $(2,m,k).$  Orientably-regular hypermaps and orientably-regular maps are those hypermaps and maps associated with regular Belyi pairs, and hence in correspondence with finite index normal subgroups of triangle Fuchsian groups of signature $(n,m,k)$ and $(2, m,k)$ respectively.

\s

A renewed interest has arisen in the last decades in the study of these objects, due to the remarkable Grothendieck’s theory of dessins d’enfants \cite{Gro}. This theory provides a fascinating link between hypermaps and compact Riemann surfaces, complex algebraic curves defined over number fields, and the Galois theory of algebraic number fields. In particular, this theoretical framework has offered a new approach to the study the internal structure of the absolute Galois group $\mbox{Gal}(\bar{\mathbb{Q}}/\mathbb{Q})$, whose natural action on algebraic curves defined over $\bar{\mathbb{Q}}$ (which is known to be faithful \cite{GGD2}) induces an action on hypermaps and on several interrelated objects.

\s

We refer to the articles \cite{JS96}, \cite{JS94}, \cite{JS78}, \cite{Wo97} and  the book \cite{GGD} for more details concerning these objects and the fascinating relations between them.

\section{Preliminary lemmata}\label{ls}
Let $p \geqslant 7$ be a prime number. Let $S$ be a compact Riemann surface of genus $1+p^2$ endowed with a group of conformal automorphisms $G$ of order $5p^2$ and let $\pi: S \to S/G$ denote the associated regular covering map. 

\begin{lemm} \label{l1}  $(S, \pi)$ is a regular Belyi pair of type $(5,5,5)$ and $G$ is nonabelian. \end{lemm}

\begin{proof}
Let $h$ denote the genus of $S/G.$ Assume that $\pi$ ramifies over $t$ values, and that the corresponding $\pi$-fibers have $G$-stabilisers of orders $m_1, \ldots, m_t$. The Riemann-Hurwitz formula applied to $\pi$ implies \begin{equation}\label{felix}2(1+p^2)-2=5p^2(2h-2+ E) \, \mbox{ where } \,  E=\Sigma_{i=1}^t (1-\tfrac{1}{m_i}).\end{equation}As each $m_i \geqslant 5$ and $t>0$ one has that  $E \geqslant \tfrac{4}{5}$. Now, if $h\geqslant 1$ then \eqref{felix} implies that $E \leqslant \frac{2}{5}$; a contradiction. Thus $h=0$ and therefore $E=\frac{12}{5}$ and $t=3$. Note that this also implies that  $m_1=m_2=m_3 = 5,$ and therefore $(S, \pi)$ is a regular Belyi pair of type $(5,5,5)$. Let $\theta: \Delta(5,5,5)\to G$ be a ske representing the action of $G$ on $S$. Observe that the surjectivity of $\theta$ implies that two elements of order 5 generate $G$, and therefore $G$ is necessarily nonabelian.
\end{proof}

\begin{lemm} \label{l2}If the conformal automorphism group $\mbox{Aut}(S)$ of $S$ is not isomorphic to $G$ and $\tilde{\pi}: S \to S/\mbox{Aut}(S)$ denotes the associated regular covering map, then $(S,\tilde{\pi})$ is a regular Belyi pair of type $(2,5,10)$ and $\mbox{Aut}(S)$ is isomorphic to a semidirect product $G \rtimes \mathbb{Z}_2$. 
\end{lemm}

\begin{proof} Assume that $\Aut(S)$ is not isomorphic to $G$. The covering map  $S/G\to S/\Aut(S)$ yields a finite index inclusion of triangle Fuchsian groups $\Delta( 5,5,5)<\Gamma$. All the inclusions between triangle Fuchsian groups have been determined by Singerman in \cite{Sing72} (see also \cite{Sing70}). In this case, the possibilities for $\Gamma$ are given in the following diagram:\begin{center}
\begin{tikzpicture}
\node (Q1) at (0,0) {$\Delta(2,3,10)$};
\node (Q2) at (1,1) {$\Delta(2,5,10)$};
\node (Q3) at (0,2) {$\Delta(5,5,5)$};
\node (Q4) at (-1,1) {$\Delta(3,3,5)$};
\draw (Q1)--(Q4) node {};
\draw (Q3)--(Q4) node {};
\draw (Q2)--(Q3) node {};
\draw (Q1)--(Q3) node {};
\draw (Q1)--(Q2) node {};
\end{tikzpicture}
\end{center}Assume that the action of $G$ extends to an action with signature $(3,3,5)$. As $[\Delta(3,3,5): \Delta(5,5,5)]=3,$ it follows that there exist a group $G'$ of order $15p^2$ and a ske $\theta': \Delta(3,3,5)\to G'$ such that $S \cong \mathbb{H}/\mbox{ker}(\theta').$ Let $P$ be the normal Sylow $p$-subgroup of $G'$ and consider the composite group epimorphism $$\Delta(3,3,5) \stackrel{\theta'}{\longrightarrow} G' \stackrel{\epsilon}{\longrightarrow} G'/P \cong \mathbb{Z}_{15},$$where $\epsilon$ denotes the canonical projection. We then obtain that $\mathbb{Z}_{15}$ is generated by two elements whose third power is the identity; a contradiction. Observe that this reasoning also shows that the action of $G$ does not extend to an action of signature $(2,3,10)$. As $[\Delta(2,5,10): \Delta(5,5,5)]=2$, we then conclude that $\Aut(S)$ has order $10p^2$ and acts with signature $(2,5,10).$ In other words, $(S, \tilde{\pi})$ is a regular Belyi pair of type $(2,5,10).$ Finally, noting that 
$G$ is a normal subgroup of $\Aut(S)$ of odd order,  we conclude that $\mbox{Aut}(S) \cong G\rtimes \Z_2.$\end{proof}

As an application of the classical Sylow theorems,  nonabelian groups of order $5p^2$ exist if and only if $p \equiv \pm 1 \mbox{ mod }5,$ and they are given as follows. If $p \equiv 1 \m 5$ then $G$ is isomorphic to $$\mathbf{G}_1=\langle a,b : a^{p^2}=b^5=1, bab^{-1}=a^k\rangle$$where $k$ is a fifth primitive root of unity in $\Z_{p^2}$, or $G$ is isomorphic to $$\mathbf{G}_{u,v}=\langle a,b,c : a^p=b^p=c^5=[a,b]=1, cac^{-1}=a^{u}, cbc^{-1}=b^{v}\rangle$$where $(u,v) \in \{(s,s), (1,s), (s, s^2), (s, s^4)\}$ and 
$s$ is a fifth primitive root of unity in $\Z_p$. On the other hand, if $p\equiv -1\m 5$ then $G$ is isomorphic to$$\mathbf{G}_{0}=\langle a,b,c : a^p=b^p=c^5=[a,b]=1, cac^{-1}=b^{t_2}, cbc^{-1}=(ab)^{t_1}\rangle$$
where $t_1$ and $t_2$ are the solutions of $y^2+y-1\equiv 0 \m p$. 

\s

Hereafter, we consider the triangle Fuchsian group $\Delta(5,5,5)$ with its canonical presentation $$\Delta(5,5,5)=\langle x_1, x_2, x_3 : x_1^5=x_2^5=x_3^5=x_1x_2x_3=1\rangle.$$

\begin{lemm} There is no compact Riemann surface of genus $1+p^2$ endowed with a group of conformal automorphisms isomorphic to $\mathbf{G}_{1,s}$ nor to $\mathbf{G}_{s,s}.$
\end{lemm}

\begin{proof} We argue by contradiction, assuming that such groups do act in genus $1+p^2.$ By Lemma \ref{l1}, the signature of the action is $(5,5,5)$ and therefore each action of $G \in \{\mathbf{G}_{1,s}, \mathbf{G}_{s,s}\}$ is represented by a ske \begin{equation}\label{gato}\theta:\Delta(5,5,5) \to G \mbox{ given by } x_1\mapsto a^{l_1}b^{m_1}c^{n_1},\, x_2\mapsto a^{l_2}b^{m_2}c^{n_2},\, x_3\mapsto a^{l_3}b^{m_3}c^{n_3}\end{equation} 
for some $0\leqslant l_i, m_i<p$ and $ 0< n_i<5$, for $1\leqslant i\leqslant 3$. Observe that the elements of order 5 of $\mathbf{G}_{1,s}$ are $b^mc^n$ for $0\leqslant m< p$  and $0<n<5$. It follows that $l_1=l_2=l_3=0$ in \eqref{gato}, and hence $\theta$ is not surjective. We now consider the group $\mathbf{G}_{s,s}.$ As $p \neq 5$ the Sylow $5$-subgroup have order 5 and therefore, after a suitable conjugation, we may assume $l_3=m_3=0$. Observe that if $l_2=0$ or $m_2=0$ then $\theta$ is not surjective. Thereby, after an appropriate automorphism of $\mathbf{G}_{s,s}$, we can assume $l_2=m_2=1$. It follows that each action of $\mathbf{G}_{s,s}$ in genus $1+p^2$ is $\mathbf{G}_{s,s}$-equivalent to one represented by the ske $\Delta (5,5,5)\to \mathbf{G}_{s,s}$ given by $x_1\mapsto a^{-s^{n_1}}b^{-s^{n_1}}c^{n_1},\, x_2\mapsto abc^{n_2},\, x_3\mapsto c^{n_3}$ for some $n_1, n_2, n_3 \in \{1, \ldots, 5\}$. Finally, the surjectivity implies that $\mathbf{G}_{s,s}=\langle abc^{n_2}, c^{n_3}\rangle=\langle ab, c\rangle \cong \mathbb{Z}_p \rtimes \mathbb{Z}_5$, a contradcition.\end{proof}

All the above allow us to focus only on the groups $\mathbf{G}_1, \mathbf{G}_{s,s^2}, \mathbf{G}_{s,s^4}$ and $\mathbf{G}_0.$ From now on, we consider the triangle Fuchsian group $\Delta(2,5,10)$ with its canonical presentation $$\Delta(2,5,10)=\langle y_1, y_2, y_3 : y_1^2=y_2^5=y_3^{10}=y_1y_2y_3=1\rangle.$$

 \section{The results}\label{resultados}

\subsection*{The group $\mathbf{G}_1$} 

Let $p \geqslant 11$ be a prime number such that $p\equiv 1\m 5.$ Consider the group $$\hat{\mathbf{G}}_1
=\langle a,b,c \,|\, a^{p^2}=b^5=c^2=(ca)^2=[c,b]=1, bab^{-1}=a^k \rangle,$$where $k$ be a fifth root of unity in $\mathbb{Z}_{p^2}$. Note that $\mathbf{G}_1$ is a subgroup of $\hat{\mathbf{G}}_1.$

\begin{theoA}
There are precisely four pairwise non-isomorphic compact Riemann surfaces of genus $1+p^2$ endowed with a group of conformal automorphisms isomorphic to $\mathbf{G}_{1}$. The conformal automorphism group of each of them is isomorphic to 
$\hat{{\mathbf{G}}}_{1}$
\end{theoA}

\begin{proof} Let $S$ be a compact Riemann surface of genus $1+p^2$ endowed with a group of conformal automorphisms isomorphic to $\mathbf{G}_1$.
Each such an action is represented by a ske of the form $$\theta:\Delta(5,5,5) \to \mathbf{G}_1 \, \mbox{ given by } \,  x_1\mapsto a^{l_1}b^{m_1},\, x_2\mapsto a^{l_2}b^{m_2},\, x_3\mapsto a^{l_3}b^{m_3}$$where $l_i \in \{0, \ldots, p^2-1\}$ and $m_i \in \{1, \ldots, 4\}$ for $i=1,2,3.$ Since $p \neq 5$,  after a suitable conjugation we may assume $l_3=0$. Observe that if $l_2$  were not invertible in $\mathbb{Z}_{p^2}$ then $a \notin \mbox{im}(\theta)=\langle a^{l_2}, b\rangle,$ contradicting the surjectivity of $\theta.$ Now, if $\delta$ satisfies $\delta l_2 \equiv 1 \mbox{ mod } p^2$ then 
the automorphism of $\mathbf{G}_1$ given by  $a\mapsto a^{\delta}, b\mapsto b$ shows that $\theta$ is $\mathbf{G}_1$-equivalent to \begin{equation}\label{tabac}\theta_{m_1,m_2,m_3}:\Delta\to \mathbf{G}_1 \, \mbox{ given by } \, x_1\mapsto a^{-k^{m_1}}b^{m_1}, x_2\mapsto ab^{m_2} , x_3\mapsto b^{m_3}\end{equation}
where $m_1+m_2+m_3\equiv 0\m 5$. Note that each $\theta_{m_1,m_2,m_3}$ is surjective, that there are exactly twelve of them, and that $m_1=m_2$, $m_1=m_3$ or $m_2=m_3$. Thereby, one has that $S$ is isomorphic to \begin{equation}\label{sur}S_{m_1,m_2,m_3}:=\mathbb{H}/{\text{ker}(\theta_{m_1,m_2,m_3})},\end{equation}for some $m_i$ as before. By Lemma \ref{l2}, if the conformal automorphism group of $S$ is not isomorphic to $\mathbf{G}_1$ then it is isomorphic to a semidirect product \begin{equation}\label{gp2} \mathbf{G}_1 \rtimes \mathbb{Z}_2 = \langle a,b,c : a^{p^2}=b^5=c^2=1, bab^{-1}=a^k, cac=a^{s}, cbc=a^lb^m\rangle,\end{equation} and acts with signature $(2,5,10).$ Observe that $s^2\equiv 1\m p^2$ and $m^2\equiv 1\m 5;$ therefore $(s,m)=(\pm 1,\pm 1)$. Note that $(s,m)\neq (1,1)$ since $\mathbf{G}_1 \times \mathbb{Z}_2$ cannot be generated by one element of order 5 and an involution. Note, in addition, that the fact that $\pm 1-k$ is invertible in $\Z_{p^2}$ implies that, if $m\equiv -1\m 5$ then $l\equiv 0\m p^2$. In this way, we obtain that \eqref{gp2} is isomorphic to
$\langle \mathbf{G}_1,c : c^2=1, cac=a, cbc=b^{-1} \rangle,$ isomorphic to $\langle \mathbf{G}_1,c : c^2=1, cac=a^{-1}, cbc=b^{-1} \rangle,$ or there is $l \in \{0, \ldots, p^2-1\}$ such that \eqref{gp2} is isomorphic to $\langle \mathbf{G}_1,c : c^2=1, cac=a^{-1}, cbc=a^lb \rangle.$ The fact that the first two groups do not have elements of order 10 implies that \eqref{gp2} is isomorphic to $\langle \mathbf{G}_1,c : c^2=1, cac=a^{-1}, cbc=a^l b \rangle.$ Up to isomorphism, in the latter group we can choose $l=0$ and therefore   $$\mbox{Aut}(S) \cong \hat{\mathbf{G}}_1= \langle a,b,c : a^{p^2}=b^5= c^2=1, bab^{-1}=a^k, cac=a^{-1}, [c,b]=1 \rangle.$$

\s

{\it Claim 1.} There are precisely four pairwise non-isomorphic compact Riemann surfaces of genus $1+p^2$ with conformal automorphism group isomorphic to $\hat{\mathbf{G}}_1.$

\s

Each such an action of $\hat{\mathbf{G}}_1$ is represented by a ske of the form  $\Theta:\Delta(2,5,10)\to \hat{\mathbf{G}}_1$ given by $y_1\mapsto a^{s_1}c, y_2\mapsto a^{s_2}b^{t_2}, y_3\mapsto a^{s_3}b^{-t_2}c ,$ 
where $0\leqslant s_1, s_2, s_3<p^2$ y $0<t_2< 5$. As the Sylow $2$-subgroups of $\hat{\mathbf{G}}_1$ have order two, up to a suitable conjugation, we can assume $s_1=0.$ Note that $s_2$ must be invertible since otherwise $\Theta$ is not surjective. Now, after considering the automorphism of  $\hat{\mathbf{G}}_1$ given by $a\mapsto a^{\eta}, b\mapsto b, c\mapsto c$ where $\eta s_2 \equiv 1 \mbox{ mod }p^2,$ we obtain that $\Theta$ is $\hat{\mathbf{G}}_1$-equivalent to  $$\Theta_t: \Delta(2,5,10) \to \hat{\mathbf{G}}_1 \, \mbox{ given by } \, y_1\mapsto c, y_2\mapsto ab^{t}, y_3\mapsto a^{-k^{-t}}b^{-t}c$$ for $0<t<5$. These four skes are pairwise $\hat{\mathbf{G}}_1$-inequivalent, because of there is no automorphism of $\hat{\mathbf{G}}_1$ fixing $c$ and sending 
$ab^{t}$ to $ab^{t'}$ for $t\neq t'.$ Since the normaliser of $\Delta(2,5,10)$ is $\Delta(2,5,10)$ itself, the action of conjugation at the level of kernels produces orbits of length one.  Consequently, the aforementioned skes define Riemann surfaces $S_t:=\mathbb{H}/\mbox{ker}(\Theta_t)$ that are pairwise non-isomorphic, as claimed.

\s

Observe that the map given by $x_1\mapsto y_3^{-1}y_2y_3, x_2\mapsto y_2$ and $x_3\mapsto y_3^2$ defines an embedding  $\iota_1:\Delta(5,5,5) \to \Delta(2,5,10).$  The restriction of $\Theta_t$ to $\iota(\Delta(5,5,5))\cong \Delta(5,5,5)$ 
$$\Theta_t|_{\Delta(5,5,5)}: \Delta(5,5,5) \to \mbox{im}(\Theta_t|_{\Delta(5,5,5)}) \cong \mathbf{G}_1 \, \mbox{ is given by }  \, x_1\mapsto a^{-1}b^{t},x_2\mapsto ab^{t}, x_3\mapsto a^{-k^{-t}+k^{-2t}}b^{-2t}.$$Observe that the restricted ske above is $\mathbf{G}_1$-equivalent to $\theta_{t,t,-2t}$. We then conclude that $S_{t,t,-2t}\cong X_t$. If we now consider the embeddings $\iota_2:x_1\mapsto y_3^{-1}y_2y_3, x_2\mapsto y_3^2, x_3\mapsto y_2$ and $\iota_3:x_1\mapsto y_3^2,x_2\mapsto y_3^{-1}y_2y_3, x_3\mapsto y_2$ instead of $\iota_1$ and proceed  analogously, we obtain that  $S_{t,t,-2t}\cong S_{t,-2t,t}\cong S_{-2t,t,t}\cong X_t$ for each $t \in \{1,2,3,4\}.$ We then conclude that the Riemann surfaces \eqref{sur} split into four isomorphism classes, and that the conformal automorphism group of each of them is isomorphic to $\hat{\mathbf{G}}_1.$
\end{proof}

\begin{coroA1} There are precisely twelve orientably-regular hypermaps of Euler characteristic $-2p^2$ of type $\{5,5,5\}$ with orientation-preserving automorphism group isomorphic to $\mathbf{G}_{1}$. These hypermaps are supported on four non-isomorphic Riemann surfaces, with three hypermaps on each surface.
\end{coroA1}
\begin{proof} After considering the proof of Theorem A coupled with the equivalence between regular Belyi pairs of signature $(5,5,5)$ and orientably-regular hypermaps of type $\{5,5,5\}$, it suffices to verify that the skes \eqref{tabac} are pairwise $\mathbf{G}_1$-inequivalent. Observe that if two distinct such skes were $\mathbf{G}_{1}$-equivalent, then there would be a nontrivial automorphism $\Phi$ of $\mathbf{G}_{1}$ satisfying $\Phi(b^s)=b^{s'}$ and $\Phi(ab^r)=ab^{r'}$. It is easy to see that this implies that  $s=s'$ and $r=r'$. \end{proof}

\begin{coroA2}
Up to duality, there are precisely two chiral pairs of orientably-regular maps of Euler characteristic $-2p^2$ of type $\{5, 10\}$ with orientation-preserving automorphism group isomorphic to $\hat{\mathbf{G}}_{1}$. 
\end{coroA2}

\begin{proof} In view of Claim 1 in the proof of Theorem A and after considering the equivalence between regular Belyi pairs of signature $(2,5,10)$ and orientably-regular maps of type $\{5,10\}$, it is enough to verify that the maps form two chiral pairs. If they were reflexive then there would be an automorphism $\Phi$ of $\hat{\mathbf{G}}_1$ satisfying that $\Phi(ab^t)=(ab^t)^{-1}.$ This in turn would imply that there is an automorphism of $\hat{\mathbf{G}}_1$ sending $b$ to $b^{-1},$ which is impossible as it contradicts the relation $bab^{-1}=a^k.$
\end{proof}

\subsection*{The group $\mathbf{G}_{s,s^2}$}Let $p \geqslant 11$ be a prime number such that $p \equiv 1 \m 5,$ and  consider the group $$\hat{\mathbf{G}}_{s,s^2}= \langle a,b,c,d :  a^p=b^p=c^5=d^2=[a,b]=(da)^2=(db)^2=1, cac^{-1}=a^{s}, cbc^{-1}=b^{s^2},  dcd=a^{1-s}b^{1-s^2}c\rangle,$$where $s$ is a primitive fifth root of unity in $\mathbb{Z}_p.$ Note that $\mathbf{G}_{s,s^2}$ is a subgroup of $\hat{\mathbf{G}}_{s,s^2}.$

\begin{theoB}There are precisely four pairwise non-isomorphic compact Riemann surfaces of genus $1+p^2$ endowed with a group of conformal automorphisms isomorphic to $\mathbf{G}_{s,s^2}$. The conformal automorphism group of each of them is isomorphic to 
$\hat{{\mathbf{G}}}_{s,s^2}$
\end{theoB}

\begin{proof} Let $S$ be a compact Riemann surface of genus $1+p^2$ endowed with a group of conformal automorphisms isomorphic to $\mathbf{G}_{s,s^2}$. By Lemma \ref{l1}, the action of $\mathbf{G}_{s,s^2}$  is represented by a ske that, up to  $\mathbf{G}_{s,s^2}$-equivalence, we may assume to be  \begin{equation}\label{felix7} \theta_{n_1,n_2,n_3}:\Delta (5,5,5)\to \mathbf{G}_{s,s^2} \, \mbox{ given by }\, x_1\mapsto a^{-s^{n_1}}b^{-s^{2n_1}}c^{n_1},\, x_2\mapsto abc^{n_2},\, x_3\mapsto c^{n_3}\end{equation}for some $n_1, n_2, n_3 \in \{1, \ldots, 5\}$ such that $n_1+n_2+n_3 \equiv 0 \mbox{ mod }5.$ Observe that each $\theta_{n_1,n_2,n_3}$ is surjective, that there are exactly twelve of them, and that $n_1=n_2, n_1=n_3$ or $n_2=n_3.$ Thus $$S \cong X_{n_1,n_2,n_3}:= \mathbb{H}/{\text{ker}(\theta_{n_1,n_2,n_3})}$$for some $n_i$ as before. The correspondence $a\mapsto a^{-1}, \, b\mapsto b^{-1}$ and $c\mapsto a^{1-s}b^{1-s^2}c$ defines an automorphism $\Phi$ of $\mathbf{G}_{s,s^2}$ that satisfies $\Phi(g_1, g_2, g_3)=(g_2, g_1, g_2g_3g_2^{-1})$, where $g_i=\theta_{n_1, n_1, n_3}(x_i)$ for $i=1,2,3.$ It follows from \cite[Case N8]{BCC} that the action of $\mathbf{G}_{s,s^2}$ on $X_{n_1,n_1,n_3}$ extends to an action of $\hat{\mathbf{G}}_{s,s^2}$ with signature $(2,5,10).$ This fact coupled with Lemma \ref{l2} imply that
$\Aut(X_{n_1,n_1,n_3}) \cong \hat{\mathbf{G}}_{s,s^2}.$ Moreover, following \cite{BCC}, the action of  $\hat{\mathbf{G}}_{s,s^2}$ on $X_{n_1,n_1,n_3}$ is represented by the ske \begin{equation}\label{max3}\Theta_{n_1,n_1,n_3}: \Delta(2,5,10)\to \hat{\mathbf{G}}_{s,s^2} \, \mbox{ given by } \, y_1\mapsto d, y_2\mapsto abc^{n_1}, y_3\mapsto (dabc^{n_1})^{-1}.\end{equation}

{\it Claim 1.} $X_{n_1,n_1,n_3} \cong X_{n_1,n_3,n_1} \cong X_{n_3,n_1,n_1}.$

\s

Consider the embeddings of $\Delta(5,5,5)$ into $\Delta(2,5,10)$ given by $\iota_2: x_1\mapsto y_2, x_2\mapsto y_3^2, x_3\mapsto y_3^{-1}y_2y_3$ and $\iota_3:x_1\mapsto y_3^2, x_2\mapsto y_3^{-1}y_2y_3, x_3\mapsto y_2.$ The restriction of $\Theta_{n_1,n_1,n_3}$ to $\iota_j(\Delta(5,5,5))$ is a ske $\Delta(5,5,5) \cong \iota_j(\Delta(5,5,5))  \to \mathbf{G}_{s,s^2}$ that is $\mathbf{G}_{s,s^2}$-equivalent to $\theta_{n_1,n_3,n_1}$ for $j=2$ and to $\theta_{n_3,n_1,n_1}$ for $j=3.$

\s

{\it Claim 2.} The Riemann surfaces $X_{1,1,3}, X_{2,2,1}, X_{3,3,4}$ and $X_{4,4,2}$ are pairwise non-isomorphic.

\s

First, observe that if two skes \eqref{max3} were $\hat{\mathbf{G}}_{s,s^2}$-equivalent then there would exist an automorphism $\Psi$ of  $\hat{\mathbf{G}}_{s,s^2}$ that fixes $d$ and sends $abc^{n}$ to $abc^{n'}$ for $n' \neq n$. This automorphism must be given by $$a\mapsto a^x b^y, \, b\mapsto a^{z}b^{w}, \, c\mapsto a^u b^v c^t, \, d\mapsto d$$for some $0\leqslant x,y,z,w,u,v<p$ and $t$ different from 1. The relation $cac^{-1}=a^{s}$ implies that $xs(1-s^{t-1})=0$ and $ys(1-s^{2t-1})=0,$ showing that $x=0$ and $t=3.$ Besides, the relation $cbc^{-1}=b^{s^2}$ implies that  $zs^2(1-s)=0$ showing that $z=0;$ a contradiction. 
The proof of the claim follows after noticing that $\Delta(2,5,10)$ is self-normalising.
\end{proof}

\begin{coroB1} There are precisely twelve orientably-regular hypermaps of Euler characteristic $-2p^2$ of type $\{5,5,5\}$ with orientation-preserving automorphism group isomorphic to $\mathbf{G}_{s,s^2}$. These hypermaps are supported on four non-isomorphic Riemann surfaces, with three hypermaps on each surface.
\end{coroB1}
\begin{proof} As in the proof of Corllary A1, it suffices to prove that the twelve skes \eqref{felix7} are pairwise $\mathbf{G}_{s,s^2}$-inequivalent. Observe that if two of them were $\mathbf{G}_{s,s^2}$-equivalent, then there would be a nontrivial automorphism $\Phi$ of $\mathbf{G}_{s,s^2}$ of the form $$a\mapsto a^x b^y, \, b\mapsto a^{z}b^{w}, \, c\mapsto a^{u}b^{v}c^t \mbox{ for some }0\leqslant x,y,z,w,u,v<p \mbox{ and }t \in \{1,2,3,4\}$$satisfying $\Phi(abc^n)=abc^{n'}$ with $n \neq n'.$  The non-existence of such an automorphism is proved analogously to the proof of Claim 2 in Theorem B.
\end{proof}

\begin{coroB2}
Up to duality, there are exactly two chiral pairs of orientably-regular maps of Euler characteristic $-2p^2$ of type $\{5, 10\}$ with orientation-preserving automorphism group isomorphic to $\hat{\mathbf{G}}_{s,s^2}$. 
\end{coroB2}
\begin{proof} 
Let $\Theta: \Delta(2,5,10) \to \hat{\mathbf{G}}_{s,s^2}$ be a ske. Since the Sylow $2$-subgroups of $\hat{\mathbf{G}}_{s,s^2}$ have order 2, up to conjugation, one may assume that $\Theta(y_1)=d.$  Moreover, as each element of order 5 of $\hat{\mathbf{G}}_{s,s^2}$ has the form $a^xb^yc^l$, we have that $\Theta$ is $\hat{\mathbf{G}}_{s,s^2}$-equivalent to  $$y_1 \mapsto d, \, \, y_2 \mapsto a^xb^yc^l, \,\,  y_3\mapsto (da^xb^yc^l)^{-1} \mbox{ for some } x,y \in \{0, \ldots, p-1\} \mbox{ and } l \in \{1,2,3,4\}.$$ Let $\delta$ be the integer in $\{1, \ldots, p-1\}$ such that $2 \delta \equiv 1 \mbox{ mod } p$. For each $l \in \{1,2,3,4\}$ we set $$x_l := \delta (1-s^l) \, \text{ and } \, y_l:=\delta(1-s^{2l}).$$The fact that $da^xb^yc^ld=a^{2(x-x_l)}b^{2(y-y_l)}c^l$ shows that  $$\mbox{im}(\Theta)=\langle d, a^xb^yc^l, a^{2(x-x_l)}b^{2(y-y_l)} \rangle=\left\{\begin{array}{cc}
 \langle b, a^xb^y c^l, d\rangle \cong (\mathbb{Z}_p \rtimes \mathbb{Z}_5) \rtimes \mathbb{Z}_2 & \text{ if } x= x_l \text{ and }y\neq y_l\\
 \langle a, a^xb^y c^l, d\rangle \cong (\mathbb{Z}_p \rtimes \mathbb{Z}_5) \rtimes \mathbb{Z}_2&\text{ if }  x\neq x_l \text{ and }y= y_l\\
 \langle a^xb^y c^l, d\rangle \cong \mathbb{Z}_{10}&\text{ if }  x= x_l \text{ and }y= y_l
\end{array}\right.$$The surjectivity of $\Theta$ implies that $x \neq x_l$ and $y \neq y_l.$ Note that $x_l$ and $y_l$ are different from 1. Now, if $x \neq 1$ or $y \neq 1$ then the automorphism $\Psi_{x,y}$ of $\hat{\mathbf{G}}_{s, s^2}$ given by $$a \mapsto a^{\alpha}, \, b \mapsto b^{\beta}, \, c \mapsto a^{\delta(1-s)(1-\alpha)} b^{\delta(1-s^2)(1-\beta)}c, \, d \mapsto d \mbox{ where } \alpha=\tfrac{1-x_l}{x-x_l} \mbox { and } \beta=\tfrac{1-y_l}{y-y_l}$$satisfies $\Psi_{x,y}(a^xb^yc^l)=abc^l.$ Thus, we obtain that $\Theta$ is 
 $\hat{\mathbf{G}}_{s,s^2}$-equivalent to a ske \eqref{max3}. These last skes are pairwise $\hat{\mathbf{G}}_{s,s^2}$-inequivalent as seen in the proof of Claim 2 the theorem above. To finish the proof we only need to prove that these maps form two chiral pairs, but this follows directly from the fact that the correspondence $abc^l \mapsto (abc^l)^{-1}$ and $dabc^l \mapsto (dabc^l)^{-1}$ does not define an automorphism of $\hat{\mathbf{G}}_{s,s^2}$ (otherwise, there would be an automorphism sending $c$ to $c^{-1}$ contradicting the relation $cac^{-1}=a^s$).
\end{proof}

\subsection*{The group $\mathbf{G}_{s,s^5}$}Let $p \geqslant 11$ be a prime number such that $p \equiv 1 \m 5,$ and consider the group 
  $$\hat{\mathbf{G}}_{s,s^4}= \langle a,b,c,d :  a^p=b^p=c^5=d^2=[a,b]=(da)^2=(db)^2=1, cac^{-1}=a^{s}, cbc^{-1}=b^{s^4}, dcd=a^{1-s}b^{1-s^4}c\rangle,$$which contains  $\mathbf{G}_{s,s^4}$ as a subgroup.  The proof of the following theorem (and of its corollaries) follows the same ideas as the proof of Theorem B. We will only provide a brief outline of the proof, highlighting the differences with the previous results.

\begin{theoC} There are precisely two non-isomorphic compact Riemann surfaces of genus $1+p^2$ endowed with a group of conformal automorphisms isomorphic to $\mathbf{G}_{s,s^4}$. The conformal automorphism group of each of them is isomorphic to 
$\hat{{\mathbf{G}}}_{s,s^4}$.
\end{theoC}

\begin{proof} Each such an action of $\mathbf{G}_{s,s^4}$  is represented by one of the twelve surface-kernel  epimorphisms \begin{equation}\label{felix8}\theta_{n_1,n_1,n_3}: \Delta(5,5,5)\to {\mathbf{G}}_{s,s^4} \, \mbox{ given by } \, x_1\mapsto a^{-s^{n_1}}b^{-s^{4n_1}}c^{n_1},\, x_2\mapsto abc^{n_2},\, x_3\mapsto c^{n_3}\end{equation}where  $n_1, n_2, n_3 \in \{1, \ldots, 5\}$ satisfy $n_1+n_2+n_3 \equiv 0 \mbox{ mod }5.$ The rule $a\mapsto b, b\mapsto a, c\mapsto c^{-1}$ defines an automorphism of $\mathbf{G}_{s,s^4}$ that provides the following equivalences: $$\theta_{1,1,3} \cong \theta_{4,4,2}, \, \theta_{1,3,1} \cong \theta_{4,2,4}, \, \theta_{3,1,1} \cong \theta_{2,4,4}\,\, \mbox{ and } \,\,  \theta_{2,2,1} \cong \theta_{3,3,4}, \, \theta_{2,1,2} \cong \theta_{3,4, 3}, \theta_{1,2,2} \cong \theta_{4,3,3}.$$In addition, since there is no automorphism of $\mathbf{G}_{s,s^4}$ that sends $c$ to $c^2$ or to $c^3$, the skes \eqref{felix8} split into the six $\mathbf{G}_{s,s^4}$-equivalence classes given above. If we write
$$Y_{n_1,n_2,n_3}:= \mathbb{H}/{\text{ker}(\theta_{n_1,n_2,n_3})}$$then, by proceeding as in the proof of Theorem B, one sees that $\Aut(Y_{n_1,n_1,n_3}) \cong \hat{\mathbf{G}}_{s,s^4}$ and that the action of this latter group is represented by the ske \begin{equation*}\label{max6}\Theta_{n_1,n_1,n_3}: \Delta(2,5,10)\to \hat{\mathbf{G}}_{s,s^4} \, \mbox{ given by } \, y_1\mapsto d, y_2\mapsto abc^{n_1}, y_3\mapsto (dabc^{n_1})^{-1}.\end{equation*}
Moreover, analogosuly as proved in Claim 1 in the proof of Theorem B, one sees thar the Riemann surfaces $Y_{n_1,n_1,n_3}, Y_{n_1,n_3,n_1}$ and $Y_{n_3,n_1,n_1}$ are isomorphic, and therefore there are at most two non-isomorphic Riemann surfaces among $Y_{n_1, n_2, n_3}$. These classes are represented by  $Y_{1,1,3}$ and $Y_{2,2,1}$. Since there is no automorphism of  $\hat{\mathbf{G}}_{s,s^4}$ that fixes $d$ and sends $c$ to $a^xb^yc^{2}$ for some $x,y \in \{0, \ldots, p-1\}$, coupled with the fact that $\Delta(2,5,10)$ is self-normalising, we conclude that $Y_{1,1,3}$ and $Y_{2,2,1}$ are non-isomorphic.
\end{proof}

\begin{coroC1} There are precisely six orientably-regular hypermaps of Euler characteristic $-2p^2$ of type $\{5,5,5\}$ with orientation-preserving automorphism group isomorphic to $\mathbf{G}_{s,s^4}$. These hypermaps are supported on two non-isomorphic Riemann surfaces; three hypermaps lie on each surface. 
\end{coroC1}

\begin{proof}
It follows directly from the proof of Theorem C.
\end{proof}

\begin{coroC2}
Up to duality, there is exactly two orientably-regular maps of Euler characteristic $-2p^2$ of type $\{5, 10\}$ with orientation-preserving automorphism group isomorphic to $\hat{\mathbf{G}}_{s,s^4}$. They are reflexive.
\end{coroC2}

\begin{proof} By proceeding  as in the proof of Corollary B2, one sees that each ske $\Theta: \Delta(2,5,10) \to \hat{\mathbf{G}}_{s,s^4}$ is $\hat{\mathbf{G}}_{s,s^4}$-equivalent to $\Theta_{1,1,3}, \Theta_{2,2,1}, \Theta_{3,3,4}$ or $\Theta_{4,4,2}.$ Besides, the rule $a\mapsto b, b\mapsto a, c\mapsto c^{-1}, d \mapsto d$ defines an automorphism of $\hat{\mathbf{G}}_{s,s^4}$ yielding the equivalence $\Theta_{1,1,3} \cong \Theta_{4,4,2}$ and $\Theta_{2,2,1} \cong \Theta_{3,3,4}$. As there is no automorphism of  $\hat{\mathbf{G}}_{s,s^4}$ that fixes $d$ and sends $c$ to $a^xb^yc^{2}$ for some $0 \leqslant x, y < p$, we conclude that $\Theta_{1,1,3}$ and $\Theta_{2,2,1}$ are $\hat{\mathbf{G}}_{s,s^4}$-inequivalent, proving the first statement.  To prove that these maps are reflexive, it suffices to notice that the automorphisms of $\hat{\mathbf{G}}_{s,s^4}$ given by
$$\Phi_1: a \mapsto b^{-s}, \, b \mapsto a^{-s^4}, \, c \mapsto c^{-1}, \, d \mapsto a^{-1-s^4}b^{-1-s}d \,\, \mbox{ and } \,\, \Phi_2: a \mapsto b^{-s^2}, \, b \mapsto a^{-s^3}, \, c \mapsto c^{-1}, \, d \mapsto a^{-1-s^3}b^{-1-s^2}d$$are involutions that satisfy $\Phi_l(abc^l)= (abc^l)^{-1}$ and $\Phi_l(dabc^l)=(dabc^l)^{-1}$ for $l=1, 2.$\end{proof}

\subsection*{The group $\mathbf{G}_0$} Let $p \geqslant 19$ be a prime number such that $p\equiv -1 \m 5$. Let $t_1$ and $t_2$ be the integers in $\{1, \ldots, p-1\}$ that solve the equation $y^2+y-1\equiv 0\m p,$ and consider the group $$\hat{\mathbf{G}}_{0}= \langle a,b,c,d \,|\, a^p=b^p=c^5=d^2=[a,b]=(da)^2=(db)^2=1, cac^{-1}=b^{t_2}, cbc^{-1}=(ab)^{t_1}, dcd=a^{t_1^2}b^{2}c\rangle.$$For later use, we recall that $t_1t_2 \equiv t_1+t_2 \equiv -1 \m p.$ Note that $\mathbf{G}_0$ is a subgroup of $\hat{\mathbf{G}}_{0}.$
\begin{theoD}
 There are precisely two non-isomorphic compact Riemann surfaces of genus $1+p^2$ endowed with a group of conformal automorphisms isomorphic to $\mathbf{G}_{0}$. The conformal automorphism group of each of them is isomorphic to 
$\hat{{\mathbf{G}}}_{0}$.
\end{theoD}

\begin{proof}
Let $S$ be a compact Riemann surface of genus $1+p^2$ endowed with a group of conformal automorphisms isomorphic to $\mathbf{G}_{0}$. The action of $\mathbf{G}_{0}$ is represented by a ske $\theta: \Delta(5,5,5)\to \mathbf{G}_{0}$ that, up to $\mathbf{G}_{0}$-equivalence, we may assume to be of the form \begin{equation}\label{ojo1}\Delta(5,5,5)\to \mathbf{G}_{0} \, \mbox{ given by } \,  x_1\mapsto a^{l_1}b^{m_1}c^{n_1},x_2\mapsto a^{l_2}b^{m_2}c^{n_2},x_3\mapsto c^{n_3},\end{equation}where $l_i, m_i \in \{0, \ldots, p-1\}$ and $n_i \in \{1,\dots, 5 \}$ satisfy $n_1 +n_2 +n_3\equiv 0\m 5$. The surjectivity of \eqref{ojo1} implies that $l_2 \neq 0$ or $m_2 \neq 0.$ Observe that, after considering the automorphism of $\mathbf{G}_{0}$ given by $a \mapsto ab, b \mapsto a^{t_1-1}1b^{t_1}, c \mapsto c$, we can further assume that both $l_2$ and $m_2$ are different from 0. Now, if $\delta \in \{1, \ldots, p-1\}$ satisfies $\delta l_2 \equiv 1 \mbox{ mod }p$ then we apply the automorphism of $\mathbf{G}_{0}$ given by $a \mapsto a^{\delta}, b \mapsto b^{\delta}, c \mapsto c$ to obtain that $\theta$ is $\mathbf{G}_{0}$-equivalent to\begin{equation}\label{ojo3}\Delta(5,5,5)\to \mathbf{G}_{0} \, \mbox{ given by } \,  x_1\mapsto a^{l_1}b^{m_1}c^{n_1},x_2\mapsto ab^{m_2}c^{n_2},x_3\mapsto c^{n_3}.\end{equation}The image of \eqref{ojo3} is $\langle ab^{m_2}, c\rangle$ and therefore it is surjective if and only if $m_2^2+t_2^2-m_2 \not\equiv 0 \m p.$ Now, if $m_2 \neq 1$ then the automorphism $\Phi_{m_2}$ of $\mathbf{G}_{0}$ given by $a\mapsto a^{1+t_1^2(m_2-1)}b^{m_2-1}, b\mapsto a^{t_1^2(1-m_2)}b, c\mapsto c$ satisfies $\Phi_{m_2}(ab^{m_2})=ab.$ We then conclude that $\theta$ is $\mathbf{G}_{0}$-equivalent to\begin{equation}\label{ojo5} \theta_{n_1, n_2, n_3}: \Delta(5,5,5)\to \mathbf{G}_{0} \, \mbox{ given by } \,  x_1\mapsto a^{l}b^{m}c^{n_1}, x_2\mapsto abc^{n_2}, x_3\mapsto c^{n_3}.\end{equation}Note that $l$ and $m$ are uniquely determined by $n_1.$ It follows that  $$S\cong Z_{n_1,n_2,n_3}:=\mathbb{H}/{\text{ker}(\theta_{n_1,n_2,n_3})}$$for some $n_i$ as before. The correspondence $a\mapsto a^{-1}$, $b\mapsto b^{-1}$, $c\mapsto a^{{t_1}^2}b^2 c$ define an automorphism $\Phi$ of $\mathbf{G}_0$ which satisfies $\Phi(g_1,g_2,g_3) = (g_2, g_1, g_2 g_3 {g_2}^{-1})$, where $g_i = \theta_{n_1,n_2,n_3}(x_i)$ for $i =1,2,3$. We now apply the results of \cite[Case N8]{BCC} and Lemma \ref{l2} to conclude that
$\Aut(Z_{n_1, n_1, n_3})\cong \hat{\mathbf{G}}_0,$ and that the action of this latter group is represented by the ske  \begin{equation*} \label{ojo10}\Theta_{n_1,n_1,n_3}: \Delta(2,5,10)\to \hat{\mathbf{G}}_0 \, \mbox{ given by } \,  y_1\mapsto d, y_2\mapsto abc^{n_1}, y_3\mapsto (dabc^{n_1})^{-1}.\end{equation*}We argue as in the previous theorems to ensure that $Z_{n_1,n_1,n_3}, Z_{n_1,n_3,n_1}$ and $Z_{n_3,n_1,n_1}$ are isomorphic. 

\s

{\it Claim 1.} $Z_{1,1,3}\cong Z_{4,4,2}$, $Z_{2,2,1}\cong Z_{3,3,4}$  and $Z_{1,1,3}\not\cong Z_{2,2,1}$.

\s
The automorphism of $\mathbf{G}_0$ given by $a\mapsto ab$, $b\mapsto b^{-1}$, $c\mapsto c^{-1}$ shows that 
 $\theta_{1,1,3}$ is $\mathbf{G}_0$-equivalent to $\theta_{4,4,2}$, and that $\theta_{2,2,1}$ is $\mathbf{G}_0$-equivalent to $\theta_{3,3,4}$. To prove the last statement of the claim, we proceed by contradiction and assume that $Z_{1,1,3}\cong Z_{2,2,1}$. Since the Fuchsian group $\Delta(2,5,10)$ is maximal, the skes $\Theta_{1,1,3}$ and $\Theta_{2,2,1}$ are $\hat{\mathbf{G}}_0$-equivalent, say, by an automorphism $\Phi$. Observe that $\Phi(abc)=abc^2.$ In particular, this implies that there exists an automorphism of $\Psi$ of $\hat{\mathbf{G}}_0$ sending $c$ to $c^2.$ If we write $\Psi(a)= a^\alpha b^\beta$ and  $\Psi(b)= a^{\gamma}b^{\delta}$ for $0\leqslant \alpha
,\beta,\gamma,\delta<p$ then the equalities $\Psi(cac^{-1})=\Psi(b^{t_2})$ and $\Psi(cbc^{-1})=\Psi
((ab)^{t_1})$ turn into to the following equations: $$(1) \, \gamma\equiv \alpha t_1 -\beta t_1^3 \m p, \,\, (2) \, \delta\equiv \alpha t_1 +\beta t_1^2\m p, \,\, (3) \, \gamma t_2+\delta t_1\equiv \alpha+\gamma\m p, \,\, (4)\, \gamma t_2\equiv \beta +2\delta\m p,$$
By substituting (1) and (2) into (3), we obtain  that $
\beta t_1^2(1+2t_1)\equiv \alpha (2+t_1-t_1^2)\equiv \alpha (1+2t_1)\m p$ and therefore, as $1+2t_1\not \equiv 0\m p$, we have that $\beta t_1^2\equiv\alpha\m p$. Similarly, by substituting (1) and (2) into (4), we obtain $-\beta (1+t_1^2)\equiv \alpha(1+2t_1)\m p$. Therefore, $\beta t_1^2(1+2t_1)\equiv -\beta (1+t_1^2)\m p.$
Note that $\beta\not \equiv 0\m p$, since otherwise $\alpha\equiv \beta\equiv 0\m p$.  It follows that $t_1^2(-1-2t_1)\equiv 1+t_1^2\m p$. Equivalently $ 2t_1+1\equiv 0\m p$, a contradiction.
\end{proof}

\begin{coroD1}
There are precisely six orientably-regular hypermaps of Euler characteristic $-2p^2$ of type $\{5,5,5\}$ with automorphism group isomorphic to $\mathbf{G}_{0}$. These hypermaps are supported on two non-isomorphic Riemann surfaces; three hypermaps lie on each surface. 

\end{coroD1}
\begin{proof} By Theorem D, we only need to verify that the skes \eqref{ojo5} form six $\mathbf{G}_0$-equivalence classes. This can be done by employing the same argument used to prove Claim 1 in the proof of Theorem D.
\end{proof}

Before proving the following corollary, we need a technical lemma.

\begin{lemm}\label{ence}
 Consider the matrices in $\mbox{M}(2, \mathbb{Z}_p)$ given by $$M_{\alpha,\beta,1}:=\begin{pmatrix}
\alpha-\delta t_1^2 & \beta-1 \\
\beta-1 & -\alpha t_2^2+\beta-\delta
\end{pmatrix}\,\,\,\,\,\,\,\,\, M_{\alpha,\beta,2}:= \begin{pmatrix}
\alpha-\delta t_1^2-t_1 & \beta-1-\delta t_1 \\
\beta -1 -\delta t_1 & \beta-\alpha t_2^2-\delta-t_2-\delta t_1
\end{pmatrix}$$where  $\alpha, \beta, \delta \in \{0, \ldots, p-1\}$ and $2 \delta \equiv 1 \mbox{ mod }p$. Then  $$\mbox{det}(M_{\alpha,\beta,l}) \equiv 0 \mbox{ mod }p \, \implies \, \langle a^{\alpha} b^{\beta}c^l, d\rangle \neq \mathbf{G}_0' \,\, \mbox{ for } \, \, l=1,2.$$
\end{lemm}

\begin{proof}
We assume $\mbox{det}(M_{\alpha,\beta,l}) \equiv 0 \mbox{ mod }p$. We write $$v:=a^\alpha b^\beta c^l, \,\,  h:=v(dvd)^{-1} \, \mbox{ and } \, H:=\langle d, v\rangle.$$Note that $h$ is the identity or has order $p.$ 

\s

{\it Claim 1.} The subgroup $\langle h, vh v^{-1}\rangle$ of $\mathbf{G}_0'$  is trivial or is isomorphic to $\Z_p$. 

\s

Assume $l=1.$ If $\beta=1$ then the fact that $\text{det}(M_{\alpha,\beta, 1})=0$ implies that $\alpha=\delta t_1^2$. It follows that $[v,d]=1$ and therefore $h$ is the identity. If $\beta\neq 1$ then we can write\begin{align*}
\langle h, vhv^{-1}\rangle&=\langle a^{2\alpha-t_1^2}b^{2\beta -2}, a^{2t_1(\beta -1)}b^{2t_1(\beta -1)+ t_2(2\alpha- t_1^2)}\rangle\\ &= \langle a^{2\alpha-t_1^2}b^{2\beta-2}, a^{2\alpha-t_1^2}b^{(2\alpha-t_1^2)(2t_1(\beta-1))^{-1}(2t_1(\beta -1)+ t_2(2\alpha- t_1^2))}\rangle.
\end{align*}The assumption $\text{det}(M_{\alpha,\beta, 1})=0$ implies$$2\beta -2\equiv (2\alpha-t_1^2)(2t_1(\beta-1))^{-1}(2t_1(\beta -1)+ t_2(2\alpha- t_1^2)) \m p,$$and consequently $\langle h, vhv^{-1}\rangle$ is  isomorphic to $\Z_p.$ We now assume $l=2$ and argue  similarly. If $\beta=1+\delta t_1$ then $\text{det}(M_{\alpha,\beta, 2})=0$ implies that $\alpha=\delta t_1^2+t_1$ and therefore $h$ is the identity. If $\beta\neq 1+\delta t_1$ then $$(2\beta-1+t_2)^{-1}(t_1(2\beta-1)-1+2\alpha t_2+t_2^2)\equiv t_1(2\beta-1)-1 \m p,$$ and hence
\begin{align*}
&\langle h, vhv^{-1}\rangle=\langle a^{2\alpha+t_2}b^{2\beta-1+t_2}, a^{t_1(2\beta-1)-1}b^{t_1(2\beta -1)-1+2\alpha t_2+t_2^2} \rangle=\\ &\langle  a^{(2\beta-1+t_2)^{-1}(t_1(2\beta-1)-1+2\alpha t_2+t_2^2)(2\alpha+t_2)}b^{t_1(2\beta-1)-1+2\alpha t_2+t_2^2}, a^{t_1(2\beta-1)-1}b^{t_1(2\beta -1)-1+2\alpha t_2+t_2^2} \rangle \end{align*}is isomorphic to $\Z_p.$ This proves the claim. 

\s

If we write $h_i:=v^i(dvd)^{-i}$ for $1\leqslant i \leqslant 4$ (note that $h_1=h$) then we have that $$
h_i=v^i(dvd)^{-i}=v\big[v^{i-1} (dvd)^{-(i-1)}\big] (dvd)^{-1}=vh_{i-1}(dvd)^{-1}=(vh_{i-1}v^{-1})h_1$$for $i=2,3,4.$ It follows from the claim above that $h_i \in \langle h\rangle$ for each $i.$ 

\s

{\it Claim 2.} If $\langle h \rangle$ is nontrivial then it is the unique subgroup of $H$ of order $p$. 

\s

Assume $(\alpha,\beta,l)\not \in \{(\delta t_1^2,1,1),(\delta t_1^2+t_1,1+\delta t_1,2)\}$ since otherwise $\langle h \rangle$ is trivial. Let $g$ be an element of $H$ of order $p.$ It follows that $g$ can be writen in one of the following  ways:$$g=(v^{m_1}d) \cdots (v^{m_k}d) \,\,\, \mbox{ or } \,\,\, g=(dv^{m_1})\cdots (dv^{m_k}) \mbox{ for an even integer $k$, or}$$
 $$g=v^{m_1}(dv^{m_2})\cdots (dv^{m_k}) \,\,\, \mbox{ or } \,\,\, g=d(v^{m_1}d)\cdots (v^{m_k}d)\mbox{ for an odd integer $k$.}$$In each case, the equality $\Sigma_{j=1}^k m_j\equiv 0\m p$ holds. By applying an inductive argument, it can be seen that each $g$ as before can be written as a product of powers of $h_1, h_2, h_3, h_4$ (for instance, in the first case $g$ is equal to $h_{m_1}h_{m_1+m_2}^{-1}h_{m_1+m_2+m_3}\cdots h_{m_1+\dots +m_{k-1}}$) and consequently $g \in \langle h \rangle.$ In other words, $\langle h \rangle$ contains all the elements of $H$ of order $p$, and the claim follows. 

\s

Finally, observe that if $\langle h \rangle$ is trivial then 
$[d,v]=1$ and thus $H\cong \Z_{10}$ and, in particular,  $H \neq \mathbf{G}_0'$. If $\langle h \rangle$ is nontrivial, then $H$ has a unique subgroup of order $p$, showing that $H\neq \mathbf{G}_0'$. 
\end{proof}

\begin{coroD2}
Up to duality, there are precisely two orientably-regular maps of Euler characteristic $-2p^2$ of type $\{5, 10\}$ with automorphism group isomorphic to $\hat{\mathbf{G}}_{0}$. They are reflexive.
\end{coroD2}

\begin{proof} Let $\Theta: \Delta(2,5,10) \to \hat{\mathbf{G}}_{0}$ be a ske. Observe that $\Theta$ is given by $$y_1\mapsto a^r b^s d, \, y_2\mapsto a^{\alpha} b^{\beta} c^l \mbox{ and }y_3\mapsto (a^r b^s d a^{\alpha}b^{\beta} c^{l})^{-1} ,$$ for some $r,s,\alpha, \beta\in \{0,\dots
p-1\}$  and $l\in \{1,2,3,4\}$.
\s

{\it Claim 1. } $\Theta$ is $\hat{\mathbf{G}}_{0}$-equivalent to $\Theta_{1,1,3}$ or  $\Theta_{2,2,1}$.

\s

The fact that the rule $a \mapsto b^{1-t_2}, b \mapsto a, c \mapsto c^{-1}, d \mapsto b^{t_1+1}d$ is an automorphism of ${\bf G}_0'$ allows us to assume, up to $\hat{\mathbf{G}}_{0}$-equivalence, that $l=1,2$. Moreover, after a suitable conjugation, 
 we may further assume that $r=s=0$. For each $x,z\in \{0,\dots p-1\}$ not simultaneously zero, consider the transformation $\phi_{x,z}: {\bf G}_0'\to {\bf G}_0'$ defined by \begin{equation*}\label{celu}a\mapsto a^{x}b^{-zt_2^2},\ b\mapsto a^{z}b^{x+z},\, c\mapsto a^{\delta(t_1^2-xt_1^2-2z)}b^{\delta(2-z-2x)}c,\, d\mapsto d,\end{equation*}where $\delta \in \{1, \ldots, p-1\}$ satisfies $2\delta \equiv 1 \mbox{ mod }p.$  The equality $\phi_{x,z}(a^\alpha b^\beta c^l)=abc^l$ can be rewritten matricially as
 $$M_{\alpha,\beta,1}  \begin{pmatrix}
x \\ z
\end{pmatrix} =
 \begin{pmatrix}
1-\delta t_1^2  \\
0 
\end{pmatrix} \mbox{ and } M_{\alpha,\beta,2}  \begin{pmatrix}
x \\ z
\end{pmatrix}  =\begin{pmatrix}
\delta t_1^2  \\
-\delta t_1
\end{pmatrix}$$where 
 $M_{\alpha,\beta,l}$ is as in Lemma \ref{ence}. We recall that the surjectivity of $\Theta$ implies that $\langle d, a^\alpha b^\beta c^l\rangle = {\mathbf G}_0'$.  Thus, Lemma \ref{ence} implies that the aforementioned linear systems have a solution, showing that there are $x_l$ and $z_l$ such that $\phi_{x_l,z_l}(a^\alpha b^\beta c^l)=abc^l$ for $l=1,2$. It follows that, to prove our claim, we only need to verify that $\phi_{x_l, z_l}$ is in fact an automorphism of  ${\bf G}_0'$. To accomplish that it suffices to prove that $$P_l:=\langle \phi_{x_l, z_l}(a), \phi_{x_l, z_l}(b) \rangle \cong \mathbb{Z}_p^2.$$ For the case $l=1$ we have that \[  
\begin{pmatrix}
x_1 \\ z_1
\end{pmatrix} = \frac{1}{\text{det}(M_{\alpha,\beta,1})}\begin{pmatrix}
(-\alpha t_2^2+\beta -\delta)(1-\delta t_1^2) \\ (1-\beta)(1-\delta t_1^2)
\end{pmatrix} .
\]
If $z_1=0$ then $P_1= \langle a^{x_1}, b^{x_1}\rangle$ and $\beta=1$.
As $\text{det}(M_{\alpha,\beta,1})\neq 0$, we have that $\alpha\neq \delta t_1^2$ and therefore $x_1\neq 0$. Thus,  $P_1= \langle a, b\rangle\cong \Z_p^2.$ 
Now, if $z_1\neq 0$ then $$P_1= \langle a^{x_1}b^{-z_1t_2^2} , a^{z_1} b^{x_1+z_1}\rangle = \langle a^{x_1}b^{-z_1t_2^2} , a^{x_1} b^{x_1z_1^{-1}(x_1+z_1)}\rangle.$$The fact that $\text{det}(M_{\alpha,\beta,1})\neq 0$ implies that $-z_1^2t_2^2\nequiv x_1(x_1+z_1) \mbox{ mod }p$ and hence $P_1 \cong  \Z_p^2.$ Similarly, for the case $l=2$ we have that \[  
\begin{pmatrix}
x_2 \\ z_2
\end{pmatrix} = \frac{1}{\text{det}(M_{\alpha,\beta,2})}\begin{pmatrix}
-\delta t_1^2-\delta^2t_1^3-\delta\alpha +\delta\beta \\ t_1^2+\delta t_1^3-\delta t_1\alpha-\delta t_1^2\beta
\end{pmatrix}.
\]
Observe that either $x_2\neq 0$ or $z_2\neq 0$, because otherwise $\beta=1+\delta t_1$ and $\alpha= \delta t_1^2+t_1$ contradicting the fact that $\text{det}(M_{\alpha,\beta,2}) \neq 0$. Now, we argue analogously as in the previous case to conclude that $P_2 \cong  \Z_p^2.$ This ends the proof of Claim 1.

\s

The fact that $\Theta_{1,1,3}$ and $\Theta_{2,2,1}$ are $\hat{\mathbf{G}}_{0}$-inequivalent was already proved in  Claim 1 in the proof of Theorem D.  Finally, the last statement follows from the fact that 
$$a \mapsto a^{2t_1}b^{t_1-1}, b \mapsto a^{-t_1}b^{-2t_1}, c \mapsto c^{-1}, d \mapsto a^{t_1-1}b^{t_2-1}d$$is an automorphism of $\hat{\mathbf{G}}_{0}$ that satisfies 
$abc \mapsto (abc)^{-1}$ and $dabc \mapsto (dabc)^{-1}$. \end{proof}

\subsection*{Proof of the main result} As each group of order $10p^2$ has a subgroup of order $5p^2$ for $p >5,$ the proof of the main theorem follows directly from the theorems and their corollaries proved in this section.

\section{Remarks} \label{remass}We end this article with some remarks.

\s

{\bf (1)} The Riemann surfaces appearing in this article are all non-hyperelliptic, namely, they do not arise as two-fold branched covers of $\mathbb{P}^1$. If so, the automorphism group of them must have necessarily a central involution, but, for instance, the involutions in $\hat{\mathbf{G}}_{s,s^2}$ and $\hat{\mathbf{G}}_{s,s^4}$ are of the form $a^x b^y d$, and none of them is central (note that $[a^x b^y d, a]=a^{-2}\neq 1$). 

\s

{\bf (2)} As expected, the behaviour for the small primes that are uncovered by our results is irregular. The cases $p=2$ and $p=5$ yield, up to duality, only one orientably-regular map each. They have genus $5$ and  $26$ with orientation-preserving automorphism groups of order 40 and 90, respectively. These maps are labeled as R5.8 and R26.5 in the list {\it All orientably-regular maps with rotation group of order at most 400} in \cite{C}. By contrast, the case $p=3$ is not realised. 
\s

{\bf (3)}  The smallest prime covered by our theorem is  $p=11,$ namely, orientably-regular maps of genus 122 with orientation-preserving automorphism group of order 1210. These maps appear in \cite{C} as follows. The four chiral pairs are listed in {\it All chiral rotary maps on orientable surfaces with up to 1000 edges} as the first four entries in the case of 605 edges. The two reflexive ones are listed in {\it All fully regular maps on closed surfaces with up to 1000 edges,} as the second and third entries in the case of 605 edges.

\s

{\bf Data availability.} No data was used for the research described in the article.


\end{document}